\documentclass[12pt]{amsart}

\usepackage{amsmath,amsthm,amscd,amsfonts,amssymb,epic,eepic,bbm, mathabx, youngtab, ytableau}
\usepackage[pagebackref,colorlinks=true,linkcolor=blue,citecolor=blue]{hyperref}

\allowdisplaybreaks
\newtheorem{thm}{Theorem}[section]
\newtheorem{cor}[thm]{Corollary}

\newtheorem{prop}[thm]{Proposition}

\theoremstyle{remark}

\newcounter{remarkscounter}

\numberwithin{equation}{section}
\newcommand{\A}{\mathbb{A}}
\newcommand{\AI}{\mathrm{AI}}

\newcommand{\I}{\mathrm{Ind}}

\newcommand{\GL}{\mathrm{GL}}

\newcommand{\ZZ}{\mathbb{Z}}

\newcommand{\Sym}{\mathrm{Sym}}

\newcommand{\Gal}{\mathrm{Gal}}

\newcommand{\QQ}{\mathbb{Q}}

\newcommand{\lto}{\longrightarrow}

\newcommand{\CC}{\mathbb{C}}

\newcommand{\quash}[1]{}

\theoremstyle{definition}

\numberwithin{equation}{subsection}

\begin{document}
\title{Poles of triple product $L$-functions involving monomial representations}
\author{Heekyoung Hahn}
\address{Department of Mathematics\\
Duke University\\
Durham, NC 27708}
\email{hahn@math.duke.edu}



\begin{abstract}
In this paper, we study the order of the pole of the triple tensor product $L$-functions $L(s,\pi_1\times\pi_2\times\pi_3,\otimes^3)$ for cuspidal automorphic representations $\pi_i$ of $\GL_{n_i}(\A_F)$ in the setting where  one of the $\pi_i$ is a monomial representation. In the view of Brauer theory, this is a natural setting to consider. The results provided in this paper give crucial examples that can be used as a point of reference for Langlands' beyond endoscopy proposal.
\end{abstract}

\maketitle

\section{Introduction}

Let $G$ be a reductive group over a number field $F$ and let $\A_F$ be the adeles of $F$. For a given representation
\begin{equation}\label{Lmap}
 {}^LG\lto \GL_n(\CC),
\end{equation} 
the  Langlands functoriality conjectures \cite{Langlands_conj}  predict there should be a corresponding transfer of automorphic representations of $G(\A_F)$ to automorphic representations of $\GL_n(\A_F)$.
The image of functorial transfers of automorphic representations can conjectually be characterized in terms of $L$-functions. This is the crux of the Langlands' beyond endoscopy proposal \cite{Langlands_beyond}. Roughly speaking, it states that if $\pi$ is a functorial transfer from $G$, then $L(s,\pi,r \otimes \eta)$ has a pole at $s=1$ for some character $\eta : F^{\times} \backslash \A_F^{\times} \to \CC^\times$ whenever a representation $r$ detects the Zariski closure of the  image of the $L$-map \eqref{Lmap} in the sense of  \cite{Hahn-PAMS} and \cite {Hahn-RNUT} (see also \cite{HHLS} for related work). 

A great deal of work on the properties of the $L$-functions has been done when  $r=\Sym^2$ (see  \cite{Arthur}, \cite{CKPSS} and \cite{GRS} for example). There are  similar results for $r=\Lambda^2$ (see \cite{GJR1} and \cite{GJR2} for instance). Therefore the tensor product $L$-function associated to $r=\otimes^2=\Sym^2 \oplus \Lambda^2$ is  relatively well-understood. In fact, via Rankin-Selberg theory, one knows the analytic properties of the tensor product Langlands $L$-functions:
$$
L(s,\pi_1 \times \pi_2,RS)=L(s, \pi_1\times\pi_2),
$$
where $\pi_1$ (resp. $\pi_2$) is a cuspidal automorphic representation of $\GL_m(\A_F)$ (resp.  $\GL_n(\A_F)$), the $L$-function on the right is the Rankin-Selberg $L$-function and the $L$-function on the left is the Langlands $L$-function associated to
$$
RS: {}^L(\GL_m\times\GL_n)\hookrightarrow {}^L\GL_{mn},
$$
the representation induced by the usual tensor product (see also \cite{GK} for related work on  Rankin-Selberg transfers).  Apart from this case, however, at this point very little is known about the general picture, even in the case of triple tensor products to be defined below. Thus the community is in dire need of concrete examples. 

To get a glimpse of the triple tensor product  $L$-functions 
$$L(s,\pi_1\times\pi_2\times\pi_3, \otimes^3):=L(s,\pi_1\times\pi_2\times\pi_3)$$associated to the tensor product
$$
\otimes^3 : {}^L(\GL_{n_1}\times \GL_{n_2} \times \GL_{n_3}) \longrightarrow {}^L\GL_{n_1n_2n_3},
$$we impose a restriction on $\pi_3$ that it is induced from a Hecke character.

An important remark should be made here: the restriction that $\pi_3$ is a monomial representation is motivated by Brauer theory.  Brauer theory states that any representation of a finite group is a $\ZZ$-linear combination of monomial representations, that is, representations induced from characters of subgroups.  In other words, this theory indicates that monomial representations are ``enough'' to study representations with finite image. 
 
To state our results, for $i=1, 2$, we let $\pi_i$ be cuspidal automorphic representations of $A_{\GL_{n_i}}\backslash\GL_{n_i} (\A_F)$, where $A_{\GL_{n_i}}$ is the neutral component of the real points of the greatest $\QQ$-split torus in the center of $\mathrm{Res}_{F/\QQ}{\GL_{n_i}}$. Let $\chi$ be a Hecke character $\chi: A_{\mathbb{G}_m} K^{\times} \backslash \A_{K}^{\times} \to \CC^\times$, where $K/F$ is any cyclic extensions of number fields of prime degree. Denote by $\pi_{iK}$ the base changes to $K$. It is well-known from \cite[Chapter 3]{ArthurClozel} that
$\pi_{iK}$ is an automorphic representation of  $A_{\GL_{n_i}}\backslash\GL_{n_i} (\A_K)$. We study the order of the pole at $s=1$ of the triple product $L$-function
\begin{equation}\label{Lfunction}
L(s, \pi_1\times\pi_2\times \AI(\chi)),
\end{equation}
where $\AI (\chi)$ is the automorphic induction of $\chi$.

In our first result, we discuss a setting when the triple tensor product $L$-function has at most a simple pole.

\begin{thm}\label{thm:intro-irr}
Let  $K/F$ be a cyclic extension of number fields of prime degree degree, let $\chi: A_{\mathbb{G}_m} K^{\times} \backslash \A_{K}^{\times} \to \CC^\times$, and for $i=1, 2$, let $\pi_i$ be cuspidal automorphic representations of $A_{\GL_{n_i}}\backslash\GL_{n_i} (\A_F)$. Suppose that one of the $\pi_{iK}$ is a cuspidal automorphic representation of $A_{\GL_{n_i}}\backslash\GL_{n_i} (\A_K)$. Then one has that
$$
\mathrm{ord}_{s=1}L(s, \pi_1\times \pi_2\times \AI(\chi))=\begin{cases} 1 \text{ if } \pi_{1K}\cong \pi_{2K}^\vee \otimes \chi^{-1}\\
0 \text{ otherwise. }
\end{cases}
$$
\end{thm}

Next result gives a bound on the order of the pole:

\begin{thm}\label{thm:intro-ind}
Let  $K/F$ be a cyclic extension of number fields of prime degree $p$, let $\chi: A_{\mathbb{G}_m} K^{\times} \backslash \A_{K}^{\times} \to \CC^\times$, and for $i=1, 2$, let $\pi_i$ be cuspidal automorphic representations of $A_{\GL_{n_i}}\backslash\GL_{n_i} (\A_F)$. Suppose that  $\pi_1$ and $\pi_2$ both are induced from $K$. Then the order of the pole of $L(s, \pi_1\times\pi_2\times \AI(\chi))$ at $s=1$ is at most $p$.  
\end{thm}

\noindent As consequences of Theorem \ref{thm:intro-irr} and Theorem \ref{thm:intro-ind},  one can deduce when these $L$-functions do not have a pole at $s=1$ and what the exact order of its pole is. This is in Proposition \ref{prop:no-pole} and Corollary \ref{cor:exact-order}, respectively. 

We close by the introduction with a few remarks. The interplay between finite group theory and Artin $L$-functions enriched both representation theory and number theory and led to important advances, including Artin-Brauer theory.  Similarly, Langlands functoriality motivates beautiful and concrete problems in the representation theory of algebraic groups and algebraic combinatorics. Many of these problems have not received the attention they deserve, but represent fertile ground for future work.

\section{Proof of the results}\label{sec:main}

We begin this section by recalling two facts: one is that the $L$-functions are invariant under induction and the other is that $\AI(\chi)$ is automorphic \cite[Chapter 3]{ArthurClozel} for any Hecke character. Here $K/F$ is a cyclic extension of number fields of prime degree.
 
 If one of $\pi_{1K}$ and $\pi_{2K}$ is a cuspidal automorphic representation of $\GL_{n_i}(\A_K)$, then the corresponding $L$-function has at most simple pole:

\begin{thm}\label{thm:irr}
Let  $K/F$ be a cyclic extension of number fields of prime degree degree, let $\chi: A_{\mathbb{G}_m} K^{\times} \backslash \A_{K}^{\times} \to \CC^\times$, and for $i=1, 2$, let $\pi_i$ be cuspidal automorphic representations of $A_{\GL_{n_i}}\backslash\GL_{n_i} (\A_F)$. Suppose that one of the $\pi_{iK}$ is a cuspidal automorphic representation of $A_{\GL_{n_i}}\backslash\GL_{n_i} (\A_K)$. Then one has that
$$
\mathrm{ord}_{s=1}L(s, \pi_1\times \pi_2\times \AI(\chi))=\begin{cases} 1 \text{ if } \pi_{1K}\cong \pi_{2K}^\vee \otimes \chi^{-1}\\
0 \text{ otherwise. }
\end{cases}
$$
\end{thm}

\begin{proof}
Recall the basic fact that if $H$ is a finite index subgroup of $G$ and $V$ a $G$-module and $W$ be a $H$-module, then 
\begin{equation}\label{ind-res}
\I (V|_H\otimes W) \cong V\otimes \I (W).
\end{equation}

Since $L$-functions are invariant under induction, using \eqref{ind-res}, one has that
\begin{align*}
L(s, \pi_1\times\pi_2\times \AI(\chi))&=L(s, \pi_{1K}\times \pi_{2K}\times \chi)\\
&=L(s, \pi_{1K}\times (\pi_{2K}\otimes \chi)).
\end{align*}

\noindent Assume that $\pi_{1K}$ is a cuspidal automorphic representation of $A_{\GL_{n_1}}\backslash\GL_{n_1} (\A_K)$. Then by the Rankin-Selberg theory (see \cite{Cogdell} for instance), one knows that $L(s, \pi_{1K}\times (\pi_{2K}\otimes \chi))$  has a simple pole when
$$
\pi_{1K} \cong \pi_{2K}^{\vee} \otimes  \chi^{-1}
$$
otherwise holomorphic. This completes the proof.
\end{proof}

Under what conditions on $\pi_1$ and $\pi_2$ do the triple tensor product $L$-functions have higher order poles? If so, how big it can be? Clearly, from Theorem \ref{thm:irr}, both $\pi_{1K}$ and $\pi_{2K}$ have to be noncuspidal. This implies that $\pi_1$ and $\pi_2$ are both induced from $K$ \cite[Chapter 3]{ArthurClozel}. In the following result, we give an upper bound on the order of the pole of triple tensor $L$-functions involving monomials:

\begin{thm}\label{thm:ind}
Let  $K/F$ be a cyclic extension of number fields of prime degree $p$, let $\chi: A_{\mathbb{G}_m} K^{\times} \backslash \A_{K}^{\times} \to \CC^\times$, and for $i=1, 2$, let $\pi_i$ be cuspidal automorphic representations of $A_{\GL_{n_i}}\backslash\GL_{n_i} (\A_F)$. Suppose that  $\pi_1$ and $\pi_2$ both are induced from $K$. Then the order of the pole of $L(s, \pi_1\times\pi_2\times \AI(\chi))$ at $s=1$ is at most $p$.  
\end{thm}

\begin{proof}
Suppose that $\pi_1$ and $\pi_2$ are induced from $K$ and denote by $\boxplus$ the isoberic sum (see \cite[Chapter 10]{GH-book} for example). Let $\langle \sigma\rangle= \Gal(K/F)$. Then it is known from \cite[Chapter 3]{ArthurClozel} that 
$$\pi_{1K}=\boxplus_{\sigma}\sigma^j(\theta_1) \quad \text{and}\quad  \pi_{2K}=\boxplus_{\sigma}\sigma^j(\theta_2),$$
where $\theta_i$ are cuspidal automorphic representation of $A_{\GL_{n_i/p}}\backslash\GL_{n_i/p} (\A_K)$. Therefore one has that
\begin{align*}
L(s, \pi_{1K}\times (\pi_{2K}\otimes \chi))&= L(s,\, \boxplus_{\sigma}\sigma^j(\theta_1) \times (\boxplus_{\sigma}\sigma^j(\theta_2)\otimes \chi))\\
&=L(s, \, \boxplus_{\sigma}\sigma^j(\theta_1) \times \boxplus_{\sigma}(\sigma^j(\theta_2)\otimes \chi))\\
&=\prod_{j, k} L(s, \sigma^j(\theta_1) \times (\sigma^k(\theta_2)\otimes \chi))
\end{align*}
for $0\leq j, k \leq (p-1)$. Hence
\begin{align*}
\mathrm{ord}_{s=1}L(s, \pi_{1K}\times (\pi_{2K}\otimes \chi))=\sum_{j, k}\mathrm{ord}_{s=1}L(s, \sigma^j(\theta_1)\times (\sigma^k(\theta_2)\otimes \chi)).
\end{align*}
 Therefore, one has to count all the possible pairs $(j, k)$, $0\leq j, k\leq (p-1)$ satisfying
\begin{equation}\label{actual-count}
(\sigma^j(\theta_2))^\vee\otimes \chi^{-1}\cong \sigma^k(\theta_1).
\end{equation}
The key fact we use is  neither $\theta_1$ nor $\theta_2$ can be invariant under $\Gal(K/F)$. Otherwise, $\pi_1$ and $\pi_2$ would be noncuspidal, contradicting our assumptions (see \cite[Theorem 6.2, Chapter 3]{ArthurClozel}).

Consider the $p\times p$ matrix whose entries are all the possible pairs of $(j, k)$, namely
$$
A=\left[
\begin{array}{ccccc}
(0,0)&(0,1)&(0,2)&\cdots &(0, p-1)\\
(1,0)&(1,1)&(1,2) &\cdots &(1, p-1)\\
(2,0)&(2, 1)& (2,2) &\cdots & (2, p-1)\\
&&\ddots&&\\
(p-1, 0)&(p-1, 1)&(p-1, 2)&\cdots &(p-1, p-1)
\end{array}\right].
$$
We claim that no two pairs in a  column or in a  row of the matrix $A$ can satisfy \eqref{actual-count} simultaneously: Suppose that two pairs $(j, k)$ and $(j, m)$ in the $j$-th row of $A$ with $k >m$ satisfy
$$
(\sigma^j(\theta_2))^\vee\otimes \chi^{-1}\cong \sigma^k(\theta_1)\quad \text{and}\quad (\sigma^j(\theta_2))^\vee\otimes \chi^{-1}\cong \sigma^m(\theta_1).$$This would imply that $\sigma^{k-m}(\theta_1)\cong\theta_1$, which contradicts to the fact that $\theta_1$ can not be invariant under $\Gal (K/F)$. Similarly, no two pairs in the $k$-th column of $A$ satisfy \eqref{actual-count}  simultaneously. Therefore the number of pairs $(j, k)$ satisfying \eqref{actual-count} is at most $p$.
\end{proof}

The following result is to state when one trivially knows that the triple tensor $L$-functions do not have a pole at $s=1$:

\begin{prop}\label{prop:no-pole}
Let  $K/F$ be a cyclic extension of number fields of prime degree, let $\chi: A_{\mathbb{G}_m} K^{\times} \backslash \A_{K}^{\times} \to \CC^\times$, and for $i=1, 2$, let $\pi_i$ be cuspidal automorphic representations of $A_{\GL_{n_i}}\backslash\GL_{n_i} (\A_F)$. Suppose that  $\pi_1$ and $\pi_2$ both are induced from $K$ with $n_1 \neq n_2$. Then      $L(s, \pi_1\times\pi_2\times \AI(\chi))$ does not have a pole at $s=1$. 
\end{prop}

\begin{proof}
Suppose that both $\pi_1$ and $\pi_2$ are induced from $K$ and $n_1\neq n_2$. Like the arguments in the proof of Theorem \ref{thm:ind}, by \cite[Chapter 3]{ArthurClozel}, we have that
$$\pi_{1K}=\boxplus_{\sigma}\sigma(\theta_1), \quad \text{and}\quad \pi_{2K}=\boxplus_{\sigma'}\sigma'(\theta_2),$$
where $\sigma, \sigma'\in \Gal(K/F)$ and $\theta_i$ are cuspidal automorphic representation of $A_{\GL_{n_i/p}}\backslash\GL_{n_i/p} (\A_K)$. Note that $\frac{n_1}{p}\neq \frac{n_2}{p}$. Therefore 
$$\sigma(\theta_2)^\vee\otimes \chi^{-1} \ncong \sigma'(\theta_1)
$$for any $\sigma, \sigma'\in \Gal(K/F)$, and hence $L(s, \sigma(\theta_1) \times (\sigma'(\theta_2)\otimes \chi))$ does not have a pole at $s=1$ again by the theory of Rankin-Selberg \cite{Cogdell}. This completes the claim.
\end{proof}

The proof of Theorem \ref{thm:ind} above provides in fact an explicit formula of the exact order of the pole of the triple product $L$-functions in a special setting:

\begin{cor}\label{cor:exact-order}
Let $K/F$ be a cyclic extension of prime order $p$ with $\Gal (K/F)=\langle \sigma\rangle$ and let $\chi: A_{\mathbb{G}_m} K^{\times} \backslash \A_{K}^{\times} \to \CC^\times$. For $i=1, 2$, let $\pi_i$ be cuspidal automorphic representations of $A_{\GL_{n_i}}\backslash\GL_{n_i} (\A_F)$ such that $\pi_{iK}=\boxplus_{\sigma}\sigma^j(\theta_i)$ for  cuspidal automorphic representations $\theta_i$ of $A_{\GL_{n_i/p}}\backslash\GL_{n_i/p} (\A_K)$.  If $\ell$ is the number of pairs $(j, k), 0\leq j, k \leq (p-1)$ satisfying
\begin{equation}\label{strong-char-rel}
(\sigma^j(\theta_2))^\vee\otimes\chi^{-1} \cong \sigma^k(\theta_1), 
\end{equation}
then $\ell$ is the order of the pole of $L(s, \pi_1\times\pi_2\times \AI(\chi))$ at $s=1$. \qed
\end{cor}

We close this paper with a final remark. One might ask if the triple tensor product $L$-functions can ever achieve the maximum order of the pole. When $p=2$, it is known from Ikeda \cite{Ikeda} that the triple product $L$-functions have at most simple pole. Therefore, in some cases, it won't happen.


\section*{Acknowledgements}

The author is grateful to J. R. Getz for his constant support throughout this project and helping with edit of the paper.


\end{document}